\documentclass[12pt, reqno]{amsart}

\usepackage{amssymb,latexsym,enumitem}
\usepackage{pgf,tikz,ifthen,arrayjobx}
\usetikzlibrary{arrows}
\usetikzlibrary{calc}

\usepackage{wasysym} %,mathtools
\usepackage[breaklinks=true,colorlinks=true,linkcolor=blue,citecolor=red,urlcolor=blue]{hyperref}

\usepackage{relsize}
\usepackage{breqn}

\allowdisplaybreaks

\makeatletter
\@namedef{subjclassname@2020}{%
  \textup{2020} Mathematics Subject Classification}
\makeatother

\theoremstyle{plain}
\newtheorem{theorem}{Theorem}
\newtheorem{corollary}{Corollary}

\newtheorem*{theorem*}{Theorem}%[section]
\newtheorem*{corollary*}{Corollary}

\theoremstyle{definition}
\newtheorem{remark}{Remark}

\begin{document}

\title{On $r$-Equichromatic Lines with few points in $\mathbb{C}^2$}

\author{Dickson Y. B. Annor}

%\address{Department of Mathematical and Physical Sciences, La Trobe University, Bendigo, Victoria 3552 Australia}
%\email{}

 % something else?
%\vspace{15cm}

%\thanks{\vspace{1cm}}

%\thanks{\vspace{0.5cm}}

%\thanks{This is me}
%\thanks{\vspace{0.5cm}}
%\thanks{Email: d.annor@latrobe.edu.au}
\thanks{The author was supported by a La Trobe Graduate Research Scholarship.}

\begin{abstract}
  Let $P$ be a set of $n$ green and $n - k$ red points in $\mathbb{C}^2$.  A line determined by $i$ green and $j$ red points such that $i + j \ge 2$ and $|i - j| \le r$ is called \emph{r-equichromatic}. We establish lower bounds for $1$-equichromatic  and $2$-equichromatic lines. 
  In particular, we show that if at most $2n-k-2$ points of $P$ are collinear, then the  number of $1$-equichromatic lines passing through at most six points is at least $\frac{1}{4}(6n-k(k+3))$, and if  at most $\frac{2}{3}(2n - k)$ points of $P$
are collinear, then the number of $2$-equichromatic lines passing through at most four points is at least $\frac{1}{6}(10n - k(k + 5))$.

\end{abstract}

\maketitle

\section{Introduction}\label{sec:intr}
In this paper we study sets of $n$ green points and $n - k$ red points in the complex plane. Let $P$ be such a set. A line containing two or more points of $P$ is said to be \emph{determined} by $P$. A line determined by at least one green and one red point is called \emph{bichromatic}. Otherwise, it is called \emph{monochromatic}. A line determined by $i$ green and $j$ red points such that $i + j \ge 2$ and $|i - j| \le r$ is called \emph{r-equichromatic}. Note that every $1$-equichromatic line is a bichromatic line.

In \cite{8}, Purdy and Smith  studied lower bounds on the number of bichromatic lines and on the number of $1$-equichromatic lines in $\mathbb{C}^2$ and  $\mathbb{R}^2$. For brevity, we will mention only the results on $1$-equichromatic lines and we refer interested readers to \cite{7, 8} for some other results.

\begin{theorem}[Purdy and Smith \cite{8}]\label{thm: PS1}
 Let $P$ be a set of $n$ green and  $n - k$ red points
in $\mathbb{R}^2$ such that the points of $P$ are not all collinear. Let $t$ be the total number of lines determined by $P$. Then the number of $1$-equichromatic lines is
at least $\frac{ 1 }{4}(t + 2n + 3 - k(k +1))$. 
\end{theorem}

\begin{theorem}[Purdy and Smith \cite{8}]\label{thm: PS2}
 Let $P$ be a set of $n$ green and  $n - k$ red points
in $\mathbb{R}^2$ such that the points of $P$ are not all collinear. Then the number of $1$-equichromatic lines determined by at most four points is
at least $\frac{ 1 }{4}(2n + 6 - k(k +1))$. 
\end{theorem}

\begin{theorem}[Purdy and Smith \cite{8}]\label{thm: PS3}
 Let $P$ be a set of $n$ green and  $n - k$ red points
in $\mathbb{C}^2$ such that no $2n - k - 2$ points of $P$ are collinear. Then the number of $1$-equichromatic lines determined by at most five points is
at least $\frac{ 1 }{4}(6n - k(k + 3))$. 
\end{theorem}

\begin{theorem}[Purdy and Smith \cite{8}]\label{thm: PS4}
 Let $P$ be a set of $n$ green and  $n - k$ red points
in $\mathbb{R}^2$ such that the points of $P$ are not all collinear. Let $t$ be the total number of lines determined by $P$. Then the number of $1$-equichromatic lines determined by at most six points is
at least $\frac{ 1 }{12}(t + 6n + 15 - 3k(k +1))$. 
\end{theorem}

Purdy and Smith \cite{8} asked whether one can prove a tight lower bound on the number of $1$-equichromatic or bichromatic lines determined by at most four points in $\mathbb{C}^2$. This question motivated the current study. Unfortunately, the closest we have come is $2$-equichromatic lines. Table 2 in Purdy and Smith \cite{8} contains the summary of their results on $1$-equichromatic lower bounds. In that table there is a lower bound for the number of $1$-equichromatic lines determined by at most six points in $\mathbb{C}^2$, but there is no result in their paper justifying this claim. 
So, we prove a lower bound for the number of $1$-equichromatic lines determined by at most six points in $\mathbb{C}^2$. Our lower bound is the same as the one claimed by  Purdy and Smith \cite{8} . 

\section{Incidence Inequalities}\label{sec:ineq}
The main ingredients used by Purdy and Smith \cite{8} and which also will be used in the present paper, are incidence inequalities. We list some well-known incidence inequalities.
 Let $t_k$ denote the number of lines that pass through exactly $k$ points.

\begin{theorem}[Melchior's Inequality \cite{4}]
 Let $S$ be a set of $n$ non-collinear
points in the plane. Then
\begin{equation}\label{eqn:mel}
\begin{split}
  \sum_{k \ge 2} (3 - k) t_{k} \ge 3. 
\end{split}
\end{equation}
\end{theorem}

The proof for  (\ref{eqn:mel}) uses Euler's polyhedral formula. In \cite{6}, Langer proved this inequality by working with pairs $(\mathbb{P}^{2}_{\mathbb{C}}, \alpha D)$ where $\mathbb{P}^{2}_{\mathbb{C}}$ is the complex projective plane with a $\mathbb{Q}$-effective (boundary) divisor $D$ such that $(\mathbb{P}^{2}_{\mathbb{C}}, \alpha D)$ is log canonical and effective.

\begin{theorem}[Langer's Inequality \cite{6}]
Let $S$ be a set of $n$ points in $\mathbb{P}^{2}_{\mathbb{C}}$, with at most $\frac{2}{3}n$ points collinear. Then
\begin{equation*}\label{eqn:langer}
\begin{split}
\sum_{k \ge 2} k t_{k} \ge \frac{n(n+3)}{3}.    
\end{split}
\end{equation*}
\end{theorem}

\begin{theorem}[Hirzebruch's Inequality \cite{2}]
Let $S$ be a set of $n$ points in $\mathbb{P}^{2}_{\mathbb{C}}$, with at most $n-2$ points collinear. Then
\begin{equation}\label{eqn:hir1}
\begin{split}
t_2 + t_3 \ge n +  \sum_{k \ge 5} (k - 4) t_{k}.    
\end{split}
\end{equation}
\end{theorem}

\begin{theorem}[Hirzebruch's Inequality \cite{3}]
Let $S$ be a set of $n$ points in $\mathbb{P}^{2}_{\mathbb{C}}$, with at most $n-3$ points collinear. Then
\begin{equation}\label{eqn:hir}
\begin{split}
t_2 + \frac{3}{4}t_3 \ge n +  \sum_{k \ge 5} (2k - 9) t_{k}.    
\end{split}
\end{equation}
\end{theorem}

Hirzebruch's inequalities do not follow from Euler's formula as one would suspect. Instead, Hirzebruch's inequalities were derived from Bogomolov–Miyaoka–Yau inequality, a deep result in algebraic geometry and it is true for arrangement of points in the complex plane. %Note there are many variants of Hirzebruch's inequality in the literature. 

Bojanowski \cite{1} and Pokora \cite{5} used Langer's work \cite{6} to prove the following theorem.

\begin{theorem}[Bojanowski-Pokora Inequality]
Let $S$ be a set of $n$ points in $\mathbb{P}^{2}_{\mathbb{C}}$, with at most $\frac{2}{3}n$ points collinear. Then
\begin{equation}\label{eqn:bopo}
\begin{split}
t_2 + \frac{3}{4}t_3 \ge n +  \sum_{k \ge 5} (\frac{1}{4}k^2 - k) t_{k}.    
\end{split}
\end{equation}
\end{theorem}

(\ref{eqn:bopo})  is equivalent to
\begin{equation}\label{eqn:eqbopo}
\begin{split}
 \sum_{k \ge 2} (4k - k^2) t_{k} \ge 4n.    
\end{split}
\end{equation}

\begin{remark}
 One should note that these inequalities (except (\ref{eqn:mel})) were originally proved for an arrangement of lines in the complex projective plane such that $t_k$ is the number
of intersection points where exactly $k$ lines of the arrangement are incident. 
\end{remark}

\begin{remark}
Purdy and Smith \cite{8} proved Theorems~\ref{thm: PS1}, \ref{thm: PS2}  and \ref{thm: PS4} using Melchior's inequality~(\ref{eqn:mel}) and proved Theorem~\ref{thm: PS3} using Hirzebruch's inequality~(\ref{eqn:hir}).
\end{remark}

\section{Lower Bounds for Lines in $\mathbb{C}^2$}
The identities below can be found in \cite{7, 8} and will be used in this section. Let $t_{i, j}$ be the number of lines determined by $P$ with exactly $i$ green points and $j$ red points, where we always assume $i + j \ge 2$. Assume that the number of green points is $n$ and the number of red points is $n$. Then the number of bichromatics point pairs is
\begin{equation*}
\sum\limits_{\substack{i, j \ge0 \\ i + j \ge 2}} ijt_{i, j} = n^2
\end{equation*}
and the number of monochromatic point pairs is
\begin{equation*}
\mathlarger{\mathlarger{\sum}}\limits_{\substack{i, j \ge0 \\ i + j \ge 2}} \left[ \left( \begin{array}{c} i \\ 2 \end{array} \right) + \left( \begin{array}{c} j \\ 2 \end{array} \right)  \right] t_{i, j} = 2\left( \begin{array}{c} n \\ 2 \end{array} \right) = n^2 - n.
\end{equation*}

In general, if we assume that the number of green points is $n$ and the number of red points is $n - k$, then the above identities become 
\begin{equation}\label{eqn:id1}
\sum\limits_{\substack{i, j \ge0 \\ i + j \ge 2}} ijt_{i, j} = n(n - k) = n^2 - nk
\end{equation}
 and 
 \begin{equation}\label{eqn:id2} \mathlarger{\mathlarger{\sum}}\limits_{\substack{i, j \ge0 \\ i + j \ge 2}} \left[ \left( \begin{array}{c} i \\ 2 \end{array} \right) + \left( \begin{array}{c} j \\ 2 \end{array} \right)  \right] t_{i, j} = \left( \begin{array}{c} n \\ 2 \end{array} \right) + \left( \begin{array}{c} n - k \\ 2 \end{array} \right) = n^2 - n -nk + \frac{k^2 + k}{2}.
\end{equation}

We subtract (\ref{eqn:id1}) from (\ref{eqn:id2}) and then split the summation to get the following identity
\begin{equation}\label{eqn:id3}
\sum\limits_{\substack{i, j \ge0 \\ i + j \ge 2}} (i + j)t_{i, j} = \sum\limits_{\substack{i, j \ge0 \\ i + j \ge 2}} (i - j)^2 t_{i, j} + 2n - (k^2 + k).
\end{equation}

\subsection{A Lower Bound for $1$-Equichromatic lines through at most six points}\hspace*{\fill}\\ \label{sec: fivepts}

As stated before, we are not able to find the claimed result of Purdy and Smith \cite{8} on $1$-equichromatic lines through at most six points in $\mathbb{C}^2$. Below we will prove the result. 

\begin{theorem}\label{myown3}
  Let $P$ be a set of $n$ green and $n - k$ red points in $\mathbb{C}^2$ such that at most $2n - k - 2$ points of $P$
are collinear. Then the number of $1$-equichromatic lines passing through at most six points is
at least $\frac{1}{4}(6n - k(k + 3))$.
\end{theorem}

\begin{proof}
First,
we express
(\ref{eqn:hir1}) as
\begin{equation}\label{eqn:newbopo}
%\begin{split}
 -(t_{0, 2} + t_{2, 0}) - t_{1, 1} - (t_{0, 3} + t_{3, 0}) - (t_{1, 2} + t_{2, 1}) +  \sum\limits_{\substack{i, j \ge0 \\ i + j \ge 5}} \left((i + j) - 4 \right) t_{i, j} \le - (2n -k).    
%\end{split}
\end{equation}
We subtract (\ref{eqn:id1}) from (\ref{eqn:id2}) and unwind the first few terms of the summation to get
\begin{dmath}\label{eqn:newbopo1}
 (t_{0, 2} + t_{2, 0}) - t_{1, 1} + 3(t_{0, 3} + t_{3, 0}) - (t_{1, 2} + t_{2, 1}) + 6(t_{0, 4} + t_{4, 0}) - 2t_{2, 2} + \mathlarger{\mathlarger{\sum}}\limits_{\substack{i, j \ge0 \\ i + j \ge 5}} \left[ \left( \begin{array}{c} i \\ 2 \end{array} \right) + \left( \begin{array}{c} j \\ 2 \end{array} \right)  -ij \right] t_{i, j} = -n + \frac{k^2 + k}{2}. 
\end{dmath}

Adding (\ref{eqn:newbopo}) and (\ref{eqn:newbopo1}) produces
\begin{dmath}\label{eqn:newbopo2}
 - 2t_{1, 1} + 2(t_{0, 3} + t_{3, 0}) - 2(t_{1, 2} + t_{2, 1}) + 6(t_{0, 4} + t_{4, 0}) - 2t_{2, 2} + \mathlarger{\mathlarger{\sum}}\limits_{\substack{i, j \ge0 \\ i + j \ge 5}} \left[ \left( \begin{array}{c} i \\ 2 \end{array} \right) + \left( \begin{array}{c} j \\ 2 \end{array} \right)  - ij \right] t_{i, j} +  \sum\limits_{\substack{i, j \ge0 \\ i + j \ge 5}} \left((i + j)- 4 \right) t_{i, j} \le -(2n - k) - n + \frac{k^2 + k}{2}. 
\end{dmath}

Let $\alpha_{i, j}$ be the coefficient corresponding to  $t_{i, j}$ produced by the left-hand side of the inequality above. One can check that the only negative coefficients are $\alpha_{1, 1} =  \alpha_{1, 2} = \alpha_{2, 1} = \alpha_{2, 2} = -2$, and $\alpha_{2, 3} = \alpha_{3, 2} = \alpha_{3, 3} = -1$. Thus
\begin{equation*}\label{eqn:meleopo1}
\begin{split}
- 2(t_{1, 1} + t_{1, 2} + t_{2, 1} + t_{2, 2} + t_{2, 3} + t_{3, 2} + t_{3, 3}) \le \frac{-6n + k(k + 3)}{2}.    
\end{split}
\end{equation*}
The result follows immediately.
\end{proof}

\subsection{A Lower Bound for $2$-Equichromatic lines through at most four points}\hspace*{\fill}\\\label{sec: fourpts}

We now consider $2$-equichromatic lines through at most four points. To begin with, we write (\ref{eqn:eqbopo}) within our context and add that to (\ref{eqn:id3}) to get
\begin{equation}\label{eqn:our2}
\begin{split}
 \sum\limits_{\substack{i, j \ge0 \\ i + j \ge 2}} \left( 5(i + j) - (i - j)^2 - (i + j)^2\right) t_{i, j} \ge  10n - k(k +5).    
\end{split}
\end{equation}

Let $\alpha_{i, j}$ be the coefficient corresponding to  $t_{i, j}$ in  (\ref{eqn:our2}). One can check that the only positive coefficients are $\alpha_{0, 2} = \alpha_{2, 0} = 2, \alpha_{1, 1} = 6, \alpha_{1, 2} = \alpha_{2, 1} = 5$, and $\alpha_{2, 2} = 4$ and therefore,
\begin{equation*}\label{eqn:meleqbopo1}
\begin{split}
 6(t_{0, 2} + t_{2, 0} + t_{1, 1} + t_{1, 2} + t_{2, 1} + t_{2, 2}) \ge 10n - k(k + 5).    
\end{split}
\end{equation*}

This gives us the following:

\begin{theorem}\label{thm:myown2}
  Let $P$ be a set of $n$ green and $n - k$ red points in $\mathbb{C}^2$ such that at most $\frac{2}{3}(2n - k)$ points of $P$
are collinear. Then the number of $2$-equichromatic lines passing through at most four points is
at least $\frac{1}{6}(10n - k(k + 5))$.  
\end{theorem}

\section*{Acknowledgements}
The author is grateful to Michael Payne for simulating discussions that brought up \cite{8}. %his interest in this work and for his useful comments and suggestions which improved the presentation. 
%Also, I would like to thank the anonymous referees for their careful reading and valuable suggestions that helped me to improve the paper.
\\\\
Dickson Y. B. Annor\\
Department of Mathematical and Physical Sciences\\
La Trobe University, Bendigo, Victoria 3552 Australia.\\
d.annor@latrobe.edu.au\

\end{document}